\documentclass[twoside,10pt]{amsart}
\usepackage{amsmath}
\usepackage{amsfonts}
\usepackage{amssymb,enumerate}
\usepackage{amsthm}
\usepackage{hyperref}

\theoremstyle{plain}
\newtheorem{theorem}{Theorem}[section]
\newtheorem{lemma}[theorem]{Lemma}
\newtheorem{definition}[theorem]{Definition}
\newtheorem{corollary}[theorem]{Corollary}
\newtheorem{proposition}[theorem]{Proposition}
\newtheorem{remark}[theorem]{Remark}
\newtheorem{example}[theorem]{Example}

\begin{document}

\title{ On Izumi's theorem on comparison of valuations}

\author{ Mohammad Moghaddam}
\address{Mohammad Moghaddam\\
         School of Mathematics, Institute for Research in Fundamental Sciences (IPM)\\
         P. O. Box 19395-5746, Tehran, Iran}
\email{moghaddam@ipm.ir}
\thanks{This work was done while the author was a Postdoctoral Research Associate at the School of Mathematics,
        Institute for Research in Fundamental Sciences (IPM)}

\date{\today}
\begin{abstract}
  We prove that the sequence of MacLane key polynomials  constructed in \cite{Mac1} and \cite{Sp2} for a valuation extension $(K,\nu)\subset (K(x),\mu)$ is finite, provided that both $\nu$ and $\mu$ are divisorial and $\mu$ is centered over an analytically irreducible local domain $(R,\frak{m})\subset K[x]$. As a corollary, we prove  Izumi's theorem on comparison of divisorial valuations. 
  We give explicit bounds for the Izumi constant in terms of the key polynomials of the valuations. We show that this bound can be attained in some cases.
\end{abstract}

\maketitle

\bibliographystyle{amsplain}

\section{Introduction}\label{intro}{

We give a  proof of the finiteness of the sequence of  MacLane key polynomials  of the extensions of the valuations, in the case of divisorial valuations centered over an analytically irreducible domain (Theorem \ref{finiteness discretness}). 
 As a result, we are able to prove Izumi's theorem:

\begin{theorem}\label{final izumi}
  Suppose $\mu$ and $\mu'$ are two divisorial $k-$valuations of a field $K/k$   such that $K$ is the quotient field of   an analytically irreducible local domain $(R,\frak{m})\subset K$. Furthermore, suppose  that $\mu$ and $\mu'$ are centered over $(R,\frak{m})$  (with common center $\frak{m}$\footnote{One easily checks that to have an Izumi constant, the two valuations must have the same center. This will be a running hypothesis whenever we take two valuations centered in a ring $R$.}).  Then there exists a real number $c>0 $ such that   $\mu(y)<c \mu'(y)$, for $y\in R\setminus\{0\}$. Therefore, the Izumi constant of these two valuations, namely the number $c_{R}(\mu,\mu'):= {\rm sup}_{y\in R\setminus \{0\}}\{\frac{\mu(y)}{\mu'(y)}\}$, is well-defined.
\end{theorem}

 In \cite{Izu} Izumi proved  an analogous result of Theorem \ref{final izumi} in the case where $R$ is the local algebra of a point $\xi\in X$ for a reduced and irreducible complex space $(X,\mathcal{O}_X)$, and when $\mu$ is the order function in the point $\xi$ and $\mu'$ is the pullback of $\mu$ under a morphism $\phi:(Y,\eta)\to(X,\xi)$ (However, notice that in this case the mappings $\mu$ and $\mu'$ are not necessarily valuations). In \cite{Rees} Rees stated Izumi's result in an algebraic setting and proved Theorem \ref{final izumi}. In \cite{Swan} Theorem \ref{final izumi}  is proved when $(R,\frak{m})$ is an analytically irreducible excellent domain.


In Section \ref{key poly} we describe  the main results of the theory of key polynomials.

In Section \ref{section: rees and finiteness}    we  prove the finiteness of the key polynomials of the divisorial valuation extension $(K,\nu)\subset(K(x),\mu)$, provided  that   $\mu$ is centered over an analytically irreducible local domain $(R,\frak{m})\subset K[x]$.  

In Section \ref{izumi theorem} we use the theory of key polynomials to prove  Izumi's theorem (Theorem \ref{final izumi}).  From a computational point of view, an interesting question is to compute the Izumi constants. In \cite{Swan}, the Izumi constants are computed for some special examples. In \cite{Rond}, the Izumi constants are used to give bounds for  the Artin functions which arise in the Artin approximation theory.   Here, we give explicit bounds for  the Izumi constant $c(\mu,\mu')$  in terms of the key polynomials of the valuation $\mu$. We show that in certain cases this bound is equal to the Izumi constant; For example we compute $c(\mu,\mathrm{ord}_{\nu,\beta})$ for any divisorial valuation $\mu$ which extends $\nu$ (Theorem \ref{weighted izumi}.(ii)).


{\bf Acknowledgments.} I would like to thank Bernard Teissier for his helpful comments and questions,  Shahram Mohsenipour for discussions on the valuations and Tirdad Sharif for his comments. Also, I would like  to thank the referee for pointing out a serious mistake in an earlier version of this paper, his helpful comments and interesting examples.
}

\section{Valuations and key polynomials}\label{key poly}{
In this section we fix the notation and recall the main results of the theory of key polynomials.

Throughout  this section $(K,\nu)$ is a field with a valuation $\nu$  whose value group is  an ordered subgroup of the ordered group $(\mathbf{R},<)$. However, the theory presented in this section is generalized in \cite{vaq1} and \cite{Sp2}  for valuations with value groups of arbitrary rank. If we allow that the value $\nu(y)$ of some nonzero elements $y\in K$  can be $\infty$ then we say that $\nu$ is a pseudo-valuation. A $k-$valuation $\nu$ of a field $K/k$ is a valuation of $K$ such that $\nu\mid_{k^*}=0$. In this paper, all the fields $K$ that we consider are extensions of a base field $k$ and all the valuations $\nu$ of $K$ are $k-$valuations.  We consider the field extension $L=K(x)$. In the case where $x$ is transcendental over $K$ we say $L/K$ is of transcendental type, and when $x$ is algebraic over $K$ we say $L/K$ is of algebraic type.  We assume $\mu$ is a (not necessarily divisorial) valuation on $L$ extending $\nu$, i.e., $(K,\nu)\subset (L,\mu)$.
  If $L/K$ is of algebraic type, we denote the minimal polynomial of $x$ over $K$ by $P(X)$ ($X$ is a new variable transcendental over $K$), and assume $\mathrm{deg}P(X)=N$. Notice that in the algebraic type case we have $L=K[X]/(P(X))$, and every element $y\in L$ has a unique representative $y(X)\in K[X]$ of degree strictly less than $N$.  For $y\in L $ we define $\mathrm{deg} y=\mathrm{deg} y(X)$.

 The valuation ring of the valuation $\nu$ is denoted by $R_{\nu}$; It is the ring consisting of nonzero elements whose value is $\geq 0$, and the zero element.  It is a local ring with maximal ideal $\frak{m}_{\nu}=\{y\in R_{\nu}:\ \nu(y)>0 \}$. The residue field of the valuation $\nu$ is by definition equal to the field $\kappa_{\nu}=\frac{R_{\nu}}{\frak{m}_{\nu}}$.  Let $R$ be an integral domain, and suppose $\nu$ is centered on the ring $R$,  which means $R \subseteq R_{\nu}$. In this situation the center of the valuation $\nu$ over the ring $R$ is defined to be the ideal $\frak{p}=\frak{m_{\nu}}\cap R$. By defenition, in the case $\nu$ is centered on the local ring $(R,\frak{m})$ the center of the valuation over $R$ is equal to $\frak{m}$. A valuation $\nu$ with value group isomorphic to $\mathbf{Z}$, centered over a local domain $(R,\frak{m},k)$ is called divisorial if $\mathrm{tr.deg} (\kappa_{\nu}/k)=\mathrm{dim}R-1$.  For $\phi\in\mathbf{R},$ set
$$\mathcal{P}_{\phi}(R)=\{x\in R\mid \nu(x)\geq \phi\},$$
$$\mathcal{P}^{+}_{\phi}(R)=\{x\in R\mid \nu(x)> \phi\},$$
where we agree that $0\in \mathcal{P}_{\phi}$
 for all $\phi,$ since its value is larger than any $\phi,$ so that by the properties of valuations the $\mathcal{P}_{\phi}$
 are ideals of $R.$

The graded algebra associated with the valuation $\nu$ over the ring $R$ is defined as
$$\mathrm{gr}_{\nu}R=\bigoplus_{\phi\in\mathbf{R}}\mathcal{P}_{\phi}(R)/\mathcal{P}^{+}_{\phi}(R).$$

See \cite{Te1} for the foundational facts about this graded ring and its role in the local uniformization problem.

For each non-zero element $x\in R,$ there is a unique $\phi\in \mathbf{R}$ such
that $x\in \mathcal{P}_{\phi}\setminus\mathcal{P}^{+}_{\phi};$ the
image of $x$ in the quotient
$(\mathrm{gr}_{\nu}R)_{\phi}=\mathcal{P}_{\phi}/\mathcal{P}^{+}_{\phi}$
is the {\sl initial form} $\mathrm{in}_{\nu}(x)$ of $x.$

\begin{definition}\label{defiition:key polynomials}
A sequence of key polynomials for the extension $(L,\mu)$  of $(K,\nu)$, with respect to a ring $R\subseteq R_{\mu}$, is a well-ordered set $\mathbf{U}=\{U_i\}_{i\leq \alpha} \subset R$, where $\alpha$ is an ordinal number, which has the following properties:    For each $\beta\in \mathbf{R}$ the additive group $\mathcal{P}_{\beta}(R)$ is generated by all the products of the form $c\prod_{i\leq \alpha} U_{i}^{a_i}$, where $c\in K$ and $a_i=0$ except for a finite number of $i$, such that $\nu(c)+\sum_{i\leq \alpha}a_i\beta_i\geq \beta$, where   $\beta_i=\mu(U_i)$. Moreover, the set $\mathbf{U}$ is minimal with this property.
\end{definition}

 For any  $i\leq \alpha$ we define $\mathbf{U}_{[i]}^{a}:=\prod_{j\leq i} U_j^{a_j}$, where $a\in \mathbf{N}^{i}$ and $a_j=0$ except for a finite number of $j$. All the sequences of key polynomials of the extension $(L,\mu)/(K,\nu)$ that we study are sequences of key polynomials with respect to the ring $R=K[x]$. From now on, we simply call them the sequence of key polynomials of the extension $(L,\mu)/(K,\nu)$.

 Next we define a combinatorial sequence of weighted polynomials of the ring $K[x]$, called a {\sl weighted basis of $K[x]$}. We will see that the sequence of key polynomials associated to a valuation extension $(K,\nu)\subset(L,\mu)$ in \cite{Mac1}, \cite{vaq1}, and \cite{Sp2}, are also a weighted basis of $K[x]$. But the converse is not true in general (See the discussion before Theorem \ref{MacLane criterion}). However, the notion of the weighted basis simplifies the combinatorial part of the description of the extension of the valuation $(K,\nu)$.

\begin{definition} \label{weighted monomials}
 A sequence of weighted polynomials $\mathbf{U}=\{U_i\}_{i\leq \alpha}\subset K[x]$, $\alpha$ an ordinal, with weights $\omega(U_i)=\beta_i\in \mathbf{R}$, is called a weighted basis for $K[x]$ with respect to the valuation $(K,\nu)$, when it satisfies the following conditions.
\begin{itemize}
\item[(a)] For every $i\leq \alpha$ we have
\begin{equation} \label{ind key}
U_{i+1}=U_i^{m_i}+U_i^{m_i-1}f_{i,m_i-1}+\cdots+U_if_{i,1}+f_{i,0},
\end{equation}
where $\mathrm{deg}f_{i,j}<\mathrm{deg}U_i$, for $j=0,\ldots, m_i-1$. Moreover, we have $\beta_{i+1}>m_i\beta_i$.
In the case $L/K$ is of algebraic type we have $\mathrm{deg} U_i\leq N$, for $i\leq \alpha$.

\item[(b)] For every  $i\leq \alpha$ and every  $f\in L$ there exists  expansions (called $i-$adic expansions of $f$)
\begin{equation}\label{i-adic expansion}
f=\sum_{\ell}c_{\ell}\mathbf{U}_{[i]}^{a_{\ell}},
 \end{equation}
 where $c_{\ell}\in K$, $a_{\ell}\in \mathbf{N}^{i}$, $a_{\ell,j}<m_j$ for $j< i$. In the case that $j$ is a limit ordinal,  and the set $E(j)=\{j': j'<j,\ j'+\omega=j,\ m_{j'}=1\}$ is non-empty, we allow at most for one  $j'_0\leq j'<j$ that $a_{\ell,j'}<m_j$, where $j'_0$ is the first element of the well-ordered set $E(j)$.
 And in the case $L/K$ is of algebraic type we have $\sum_{j\leq i}a_{\ell,j}\mathrm{deg}U_j\leq N$.
\item[(c)] For any $i\leq \alpha$ we define a weight map $\omega_i:K[x]\to \mathbf{R}$ as follows: For any $f\in K[x]$
\begin{equation}\label{omega_i weight}
    \omega_i(f)=\mathrm{min}_{\ell}\{\nu(c_{\ell})+\sum_{j\leq i}a_{\ell,j}\beta_j\},
    \end{equation}
where $f=\sum_{\ell}c_{\ell}\mathbf{U}_{[i]}^{a_{\ell}}$, is an $i-$adic expansion of $f$. Then we have
\begin{equation}
\omega_i(f_{i,j})+j\beta_i=m_i\beta_i,
\end{equation}
 for $j=0,\ldots, m_i-1$ in the equation (\ref{ind key}). In other words, all the components in the right hand side of the equation (\ref{ind key}) are of the same $\omega_i$-weight.
\item[(d)] When  $L/K$ is of transcendental type we have  $\alpha\leq\omega$\footnote{$\omega$ is the ordinal type of the set of natural numbers.}. In this case: If $\alpha=\omega$,  either $\mathrm{deg}_{i\to\omega}U_i=\infty$, in which case we write $U_{\omega}=0$. Or, there exists a natural number $i_0$ such that for $j\geq i_0$ we have $\mathrm{deg}U_j=\mathrm{deg}U_{i_0}$, in which case we have
    \begin{equation}\label{limit key poly}
    U_{\omega}=\lim_{i\to\omega}U_{i}\in\hat{K}[x].
    \end{equation}
 The  field $\hat{K}$ is, by definition, the completion of $K$ with respect to the valuation $\nu$. Moreover, in this case $\beta_{\omega}=\lim_{i\to\omega}\beta_i=\infty$.
\item[(e)] For any $i<\alpha$ we have $\beta_{i+1}>m_i\beta_i$ (As a result, for any $j<i$ we have $\beta_{i}>(\prod_{j\leq j'<i} m_{j'})\beta_j$).

\end{itemize}
\end{definition}

The existence of the  $i-$adic expansions is a result of the Euclidean division algorithm (\cite{Mac1}, \cite{vaq1}, \cite{Sp2}, \cite{Mog3}): Given an element $f\in L$,
 by successive division of $f$ with $U_{i}$ one can write
\begin{equation}
f=U_{i}^{r}f_j+\cdots+U_i f_1+f_0,
\end{equation}
where $\mathrm{deg}f_j<\mathrm{deg}U_{i}$. Now, for any $f_j$ let $j'$ be the largest index such that $\mathrm{deg}f_j\geq \mathrm{deg} U_{j'}$, and repeat the same procedure for $f_j$ and $U_{j'}$. This produces the $i-$adic expansion of $f$  in the equation (\ref{i-adic expansion}). This algorithm shows that in the $i-$adic expansion we have $a_{\ell,i'}<m_{i'}$, for $i'<i$.

For $i\leq \alpha$ {\sl  a monomial of $i-$adic form}  is a product $c \mathbf{U}_{[i]}^{a}$, such that $a_j<m_j$, for $j<i$ and $c\in K$.
 Thus equation (\ref{i-adic expansion}) shows that every element $f\in L$ has a unique expansion in terms of the monomials of $i-$adic form.
\begin{remark}\label{remark: af wei}
\begin{itemize}
\item[(i)] We have $\mathrm{deg}U_{i+1}=m_i\mathrm{deg}U_i$.

\item[(ii)] When $L/K$ of algebraic type, in the construction of \cite{Sp2}  the key polynomial $U_{i+1}$ is obtained by lifting to $L$ of the minimal polynomial of $\mathrm{in}U_i$ (which is an element of a suitable graded ring), so, in general the key polynomials can have degree $N$.  Some times we consider the {\sl reduced form} of the $U_i$, denoted by $\overline{U}_i$, which are the unique representations of the $U_i$ of degree $\leq N-1$ (we get $\overline{U}_i$ after dividing $U_i$ by the minimal polynomial of $x$ over $K$).  Notice that if $\mathrm{deg}(U_i)<N$ then $\overline{U}_i=U_i$. For an adic expansion $\sum_i M_i(\mathbf{U}) $ we define the reduced form of the adic expansion  by replacing every $U_j$ by $\overline{U}_j$, in every adic monomial $M_i(\mathbf{U})$.
\item[(iii)] For any $j> i$ it is not necessarily true that the $j-$adic expansion of $U_i$ is itself (For example if $m_i=1$ then $U_{i+1}=U_i-f_{i,0}$, and the $(i+1)-$adic expansion of $U_i$ is equal to $U_{i+1}+f_{i,0}$); However, we have $\omega_j(U_i)=\beta_i$.
\item[(iv)] If $\mathrm{deg}f<\mathrm{deg}U_i$ then for any $j\geq i$ the $j-$adic expansion and $i-$adic expansion of $f$ are identical. Thus, we have $\omega_j(f)=\omega_i(f)$.
\item[(v)] For any $i< \alpha$ we have $\omega_i(f)<\omega_{i+1}(f)$, for any $f\in L$.
\item[(vi)] For $i\leq \alpha$  any expansion $f=\sum_{a}c_a\mathbf{U}_{[i]}^{a}$, for $a\in \mathbf{N}^{i}$ and without any restriction on $a_j$, is called an $i-$expansion of the element $f$.
\end{itemize}
\end{remark}
The main result of the theory of the key polynomials clarifies the relation between the totality of the extensions $(K,\nu)\subset (L,\mu)$ and the weighted bases of $K[x]$ with respect to the valuation $(K,\nu)$.

\begin{theorem}\label{main key poly} {\rm (\cite{Mac1}, \cite{vaq1}, \cite{Sp2}) }
    Given a valuation extension $(K,\nu)\subset (L,\mu)$, such that $(K,\nu)$ is divisorial, there exists a weighted basis $\mathbf{U}$ of $K[x]$, with weights $\omega(U_i)=\mu(U_i)$, which is at the same time a sequence of key polynomials for $\mu$. Moreover, if $L/K$ is of transcendental type (resp., if $L/K$ is of algebraic type) for this valuation the weight maps $\omega_i$ are valuations of the field $L$ (resp., of the field $K(X)$, where $X$ is a new variable) extending $(K,\nu)$; We have $\omega_1(f)<\omega_2(f)<\cdots<\omega_{\alpha}(f)$, for any $f\in L$, and we have $\mu=\omega_{\alpha}$.
\end{theorem}

We remark that Theorem \ref{main key poly} is valid without the assumption of $(K,\nu)$ being divisorial (see \cite{vaq1} and \cite{Sp2}). One of the technical subtleties of the construction of the key polynomials is in the case of key polynomials indexed with limit ordinals. Later  we show that in the course of extending divisorial valuations $(K,\nu)$ to divisorial valuations $(L,\mu)$ we do not meet limit ordinals, provided that $\nu$ is centered over an analytically irreducible domain $(R,\frak{m})\subset K$.

\begin{remark}\label{remark:af main}
We have:
\begin{itemize}
\item[(i)] In the case $L/K$ is of algebraic type  for $i<\alpha$  the valuations $(K(X),\omega_i)$ are not valuations of the field $L$ (in general). But for any element $f\in L$ if the initial of the $i-$adic expansion is equal to the initial of its $i+1-$adic expansion then for  $i\leq j\leq\alpha$ we have $\omega_j(f)=\omega_{\alpha}(f)=\mu(f)$. In other words, the $i-$adic expansion of such elements suffices to determine the value of $f$.
 \item[(ii)]The converse of theorem \ref{main key poly} is not true, i.e., {\sl it is not true that to every weighted basis $\{U_i;\beta_i\}_{i\leq \alpha}$ of $K[x]$ one can associate a valuation extension $(K,\nu)\subset(L,\mu)$ such that $\mu=\omega_{\alpha}$. } More precisely, in general the weight maps $\omega_i$ associated to a weighted basis $\mathbf{U}$ are not valuations of the field $L$.
\end{itemize}
\end{remark}

In the construction of key polynomials of \cite{Sp2},  only the last weight map will be a valuation of $L$. This gives a class of examples of weights maps which are not valuations, when $L/K$ is of algebraic type.  When $L/K$ is of transcendental type the situation seems to be different;   We have  a sufficient {\sl algebraic} condition of MacLane for weights to be valuations. But, it is not clear whether MacLane's condition  is automatically satisfied by any weighted basis or not.

\begin{theorem} \label{MacLane criterion}{ \rm (\cite{Mac1} Theorem 4.2, and \cite{vaq1}, Theorem 1.2)}
Suppose $\{U_i\}_{i\leq\alpha}$ is a weighted basis of $K[x]$ with respect to the (not necessarily divisorial)  valuation $(K,\nu)$. Suppose for some $i<\alpha$,  when $L/K$ is of transcendental type (respectively, when $L/K$ is of algebraic type) all the weight maps $\omega_i$ are valuations of the field $L$ (respectively, are valuations of the field $K(X)$). If $\mathrm{in}_{\omega_i}(U_{i+1})$ is irreducible in $\mathrm{gr}_{\omega_i}K[x]$ and of minimal degree (in the sense that if for some $f\in K[x]$ we have $\mathrm{in}_{\omega_i}(U_{i+1})\mid \mathrm{in}_{\omega_i}(f)$ then $\mathrm{deg}(f)\geq \mathrm{deg}(U_{i+1})$) then the weight map $\omega_{i+1}$ is a valuation of the field $L$, when $L/K$ is of transcendental type, and a valuation of the field $K(X)$, when   $L/K$ is of algebraic type.
\end{theorem}

A sufficient combinatorial condition for $\omega_i$ to be a valuation can be given. Suppose $\{U_i;\beta_i\}_{i\leq \alpha}$ is a weighted basis of $K[x]$ with respect to the (not necessarily divisorial) valuation $(K,\nu)$. Let $\Phi_i=(\nu(K),\beta_1,\ldots,\beta_i)\subset \mathbf{R}$ be the group generated by the first $i-$weights. Set $n_i=[\Phi_i:\Phi_{i-1}]$. Note that we must have
\begin{equation}\label{index-power}
m_i=n_ip_i,
\end{equation}
 for some $p_i\in\mathbf{N}$ ($m_i$ is defined in (\ref{ind key})) and moreover equation (\ref{ind key}) should be of the form:
\begin{equation}\label{precise ind key}
U_{i+1}=U_i^{n_ip_i}+U_i^{n_i(p_i-1)}f_{i,n_i(p_i-1)}+\cdots+U_i^{n_i}f_{i,n_i}+f_{i,0}.
\end{equation}

\begin{theorem}\label{necessary}
With the notation of the last paragraph, if $m_i=n_i$  then the weight maps $\omega_i$ are valuations.
\end{theorem}

\begin{proof} The proof of Theorem 4.7 of \cite{Mog3} can be adapted to this situation.
\end{proof}

In the special case where $K=k((y))$, where $k$ is algebraically closed, we have a complete combinatorial characterization of the weighted bases of $K[x]$ which correspond to the key polynomials. In \cite{FJ}, Chapter 3, Favre and Jonsson give an explicit construction of the key-polynomials of the $K[x]$. In our settings, they show that any weighted basis of $K[x]$ corresponds to a valuation only if $m_i=n_i$ (See \cite{FJ}, Corollary 2.5, and Theorem 2.29). Thus, we have

\begin{theorem}  In the case where $K=k((y))$, and $k$ is algebraically closed, the converse of Theorem \ref{necessary} is true, i.e., if $\{U_i;\beta_i\}_{i\leq\alpha}$ is a weighted basis of $K[x]$ such that the corresponding weight maps $\omega_i$ are valuations then $m_i=n_i$.
\end{theorem}
\begin{definition}\label{weighted valuation}
Let $\alpha$ be an ordinal. Let us denote by $K[U]$ the polynomial ring $K[U_1,\ldots,U_{\alpha}]$. We define a sequence $\tilde \omega_i$ of Gauss valuations on the fields $K(U_1,\ldots ,U_i)$ as follows: Fixing values
  $\tilde{\omega}_i(U_j)=\beta_j$, for $j\leq \alpha$,  we extend $\tilde{\omega}_i$ to   $f(U)=\sum_{\ell}c_{\ell}\mathbf{U}_{[\alpha]}^{a_{\ell}}\in K[U_1,\ldots,U_i]$ by
  $$\tilde{\omega}_{i}(f(U))=\mathrm{min}_{\ell}\{\nu(c_{\ell})+\sum_{j\leq i}a_{\ell,j}\beta_j\}.$$

\end{definition}

The valuation $\tilde{\omega}_i$ defines a weight map on the space of the $i-$expansions of the elements of $K[x]$. More precisely, if $f_1=\sum_{a} c_a \mathbf{U}_{[i]}^{a}$ is an $i-$expansion of the element $f\in K[x]$ then we define $\tilde{\omega}_i(f_1)=\tilde{\omega}_i(\sum_{a} c_a \mathbf{U}_{[i]}^{a})$. Notice that $\tilde{\omega}_i(f)$, is not well-defined for an element $f\in K[x]$.
The next theorem gives an algorithm to get the adic expansion of the elements of $K[x]$ without making divisions (Although being general, we explain the algorithm only in the case of finitely many key polynomials, which is the case we will use later).

\begin{theorem} \label{getting adic} Suppose $\{U_i;\beta_i\}_{i\leq \alpha}$, $\alpha\in\mathbf{N}$, is a weighted basis of $K[x]$ with respect to the valuation $(K,\nu)$, and let $f\in K[x]$. Then we have:
\begin{itemize}
\item[(i)] If $f_0=\sum_{a}c_{a}U^{a}$, where $a\in \mathbf{N}^{i}$, is the $i-$adic expansion of $f$ then the $(i+1)-$adic expansion of $f$ can be obtained by the following algorithm:
    \begin{itemize}
    \item[(a)] In $f_0$ replace any occurrence of $U_{i}^{m_i}$ by its $(i+1)-$adic expansion:
     \begin{equation}
     U_i^{m_i}=U_{i+1}-U_i^{m_i-1}f_{i,m_i-1}-\cdots-U_if_{i,1}-f_{i,0}.
     \end{equation}
     Suppose $f_1=\sum_{b}c_{b}U^{b}$, where  $b\in\mathbf{N}^{i+1}$, is the resulting expansion of $f$.
    \item[(b)] In $f_1$, for any $j\leq i$ replace any occurrence of $U_{j}^{m_j}$ by its $(j+1)-$adic expansion. Suppose $f_2\in K[U]$ is the resulting expansion of $f$.
    \item[(c)] Iterate step (b), as far as possible.
    \end{itemize}
\item[(ii)] If $f_0=\sum_{a}c_{a}U^{a}$, where $a\in \mathbf{N}^{i+1}$, is the $(i+1)-$adic expansion of $f$ then the $i-$adic expansion of $f$ can be obtained by the following algorithm:
    \begin{itemize}
    \item[(a)] In $f_0$ replace any occurrence of $U_{i+1}$ by its $i-$adic expansion:
     \begin{equation}
     U_{i+1}=U_i^{m_i}+U_i^{m_i-1}f_{i,m_i-1}+\cdots+U_if_{i,1}+f_{i,0}.
     \end{equation}
     Suppose $f_1=\sum_{b}c_{b}U^{b}$, where  $b\in\mathbf{N}^{i+1}$, is the resulting expansion of $f$.
    \item[(b)] In $f_1$ for any $j<i$ replace any occurrence of $U_{j}^{m_j}$ by its $(j+1)-$adic expansion. Suppose $f_2\in K[U]$ is the resulting expansion of $f$.
    \item[(c)] Iterate step (b), as far as possible.
    \end{itemize}
\end{itemize}
Both algorithms stop after a finite number of steps and in both cases they generate a sequence of expansions for $f$: $f_1,f_2,\ldots,f_t$, where $t\in \mathbf{N}$. In the case (i), $f_t$ is equal to the $(i+1)-$adic expansion of $f$, and in the case (ii), $f_t$ is equal to the $i-$adic expansion of $f$. Moreover, in the case (i), we have $\tilde{\omega}_{i+1}(f_1)\leq \tilde{\omega}_{i+1}(f_2)\leq\cdots\leq \tilde{\omega}_{i+1}(f_t)=\omega_{i+1}(f)$, where $\tilde{\omega}_{i+1}$ is the valuation of the ring $K[U_1,\ldots,U_{i+1}]$, defined in Definition \ref{weighted valuation}, and $\omega_{i+1}$ is the weight map associated to the weighted basis $\{U_i;\beta_i\}_{i\leq \alpha}$ in Definition \ref{weighted monomials}.   And in the case (ii), we have $\tilde{\omega}_{i}(f_1)\leq \tilde{\omega}_{i}(f_2)\leq\cdots\leq \tilde{\omega}_{i}(f_t)=\omega_{i}(f)$.
\end{theorem}
\noindent
\begin{proof}
The proof of Proposition 3.10 of \cite{Mog3} can be adapted to this situation (See also Lemma 6.5 of \cite{Mog2}).
\end{proof}
}
\section{Finiteness of key polynomials}\label{section: rees and finiteness}
{
There are delicate relations between valuations over a local domain $(R,\frak{m})$ and its (possible) extensions to the $\frak{m}-$adic completion $(\hat{R},m\hat{R})$. In general, such an extension need not exist and in case of the existence, such  extensions are far from being unique and the classical invariants of the extension may change in general   \cite{HeinS}, \cite{Te1}, and \cite{HOST}. However, in the case of divisorial valuations, centered over an analytically irreducible local domain, such extensions exist and are unique. The local domain $(R,\frak{m})$ is called {\sl analytically unramified} (resp., {\sl analytically irreducible}) if the $\frak{m}-$adic completion $(\hat{R},m\hat{R})$ does not contain nilpotent elements (resp., is a domain).

\begin{lemma} (\cite{Swan}, Lemma 1.1) Let $(R,\frak{m})$ be an analytically irreducible domain. Then every divisorial valuation, centered over $(R,\frak{m})$, extends naturally to a divisorial valuation centered over $(\hat{R},\frak{m}\hat{R})$, where $\hat{R}$ is the $\frak{m}-$adic completion of $R$.
\end{lemma}
\begin{proposition}
Suppose $\nu$ is a valuation of rank $1$, centered over the Noetherian local domain $(R,\frak{m})$. Assume that $\nu$ extends to a valuation $\mu$ of rank $1$, centered over the local ring $(\hat{R},\frak{m}\hat{R})$, where $\hat{R}$ is the $\frak{m}-$adic completion of $R$. Then such an extension is unique. Moreover, given any $0\neq f=\{f_i\}_{i\in \mathbf{N}}\in \hat{R}$, where $f_i$ is a Cauchy sequence in $R$, there exists $i_0\in \mathbf{N}$ such that for all $j\geq i_0$ we have $\mu(f)=\nu(f)$.
\end{proposition}
\begin{proof}
As $(R,\frak{m})$ is Noetherian $\nu(\frak{m})>0$ exists. Choose $i_0\in \mathbf{N}$ such that $i_0\nu(\frak{m})>\mu(f)\geq 0$. We have $f-f_j\in \frak{m}^j$. Thus,  for $j\geq i_0$ we have $\mu(f-f_j)>\mu(f)$, which shows that $\mu(f)=\mu(f_j)=\nu(f_j)$.
\end{proof}
The last two results give us:
\begin{corollary}\label{corollay:uniqueness of extension to completion}
Let $(R,\frak{m})$ be an analytically irreducible domain. Then for every divisorial valuation centered over $(R,\frak{m})$   there is a unique extension of $\nu$ to a divisorial valuation centered over $(\hat{R},m\hat{R})$.
\end{corollary}

Here we mention an obvious result
\begin{theorem}\label{theorem:analytical irreducibility of valuation ring} Let $\nu$ be a divisorial valuation centered over an analytically irreducible local domain $(R,\frak{m})$. Then $R_{\nu}$ is analytically irreducible.
\end{theorem}
Here, we notice an easy consequence of the commuting of  the completion with the quotient:
\begin{lemma}\label{lemma:invariance of minimal poly}
Suppose $R[x]$ is an analytically irreducible local domain, and $K$ is the quotient field of $R$. Assume that $x$ is algebraic over $K$. Let $S$ (resp., $\hat{R}$ ) be the completion of $R[x]$ (resp., the completion of $R$) with respect to their (respective) maximal ideals. And, assume that $\hat{K}$ is the quotient field of $\hat{R}$. Then, the minimal polynomial of the element $x\in R$ over $K$ is identical to the minimal polynomial of  $x\in S$ over $\hat{K}$.
\end{lemma}

Finally, we are ready to prove the finiteness result:
\begin{theorem} \label{finiteness discretness} If $(K,\nu)\subset(L,\mu)$ is a valuation extension and $\{U_i;\beta_i\}_{ i\leq \alpha}$ is a weighted basis of $K[x]$ with respect to $(K,\nu)$ such that $\mu=\omega_{\alpha}$ and $\nu$, $\mu$ be divisorial valuations. Moreover, assume that $\mu$ is centered over an analytically irreducible domain $(R,\frak{m})\subset L$.  Then  the number of key polynomials of the divisorial valuation $(L,\mu)$ is finite, i.e., we have $\alpha<\omega$.
\end{theorem}
\noindent
\begin{proof} As $(K,\nu)$ is a divisorial valuation, we have $\mathrm{dim}\nu=\mathrm{tr.deg}_k\kappa_{\nu}=\mathrm{tr.deg}_k K-1$.
Suppose we have $\alpha\geq \omega$. As the extended valuation $\mu$ is discrete, we see that $\alpha\leq \omega$ ($\beta_{\omega}\geq \lim_{i\to\omega}\beta_i=\infty$). Thus, we only need to consider the case $\alpha=\omega$. We distinguish the two cases of the transcendental and algebraic type:
 \begin{itemize}
\item If $L/K$ is of transcendental type, and $\alpha=\omega$, we show that
\begin{equation}\label{dim val divisorial}
\mathrm{dim}_k\kappa_{\mu}=\mathrm{dim}_k\kappa_{\nu}=\mathrm{tr.deg}_kL-2.
\end{equation} Hence $(L,\mu)$ cannot be a divisorial valuation (because if $\mu$ was a divisorial valuation then $\mathrm{dim}_k\kappa_{\mu}=\mathrm{tr.deg}_kL-1$). To prove equation (\ref{dim val divisorial}) first notice that as $(L,\mu)$ is an extension of $(K,\nu) $, we have $\kappa_{\nu}\subseteq \kappa_{\mu}$. It is sufficient to show that $\kappa_{\mu}$ is an algebraic extension of $\kappa_{\nu}$. Consider the natural map $\iota: R_{\mu}\to \kappa_{\mu}$. Clearly, it is sufficient to prove $\iota(\frac{c_a\mathbf{U}_{[i]}^a}{c_b\mathbf{U}_{[i]}^{b}})$ is algebraic over $\kappa_{\nu}$, for any $i<\alpha$ and $a,b\in \mathbf{N}^i$ such that  $\frac{c_a\mathbf{U}_{[i]}^a}{c_b\mathbf{U}_{[i]}^{b}}\in R_{\mu}$. We prove this by induction on $i$. Note that, without loss of generality,  we can assume $a_i>0$ and $b_i=0$.  Suppose the claim is proved for $i-1$,  we prove it for $i$. Let $M=\iota(\frac{c_a\mathbf{U}_{[i]}^a}{c_b\mathbf{U}_{[i]}^{b}})\neq 0$. Write $n_i\beta_i=\beta_0+\sum_{j<i}m_j\beta_j$, where $0\leq m_j<n_j$ and $\beta_0\in \nu(K)$, and set $A=cU_1^{m_1}\cdots U_{i-1}^{m_{i-1}}$, where $\nu(c)=\beta_0$. As $M\neq 0$, we have $\mu(c_a\mathbf{U}_{[i]}^a)=\mu(c_b\mathbf{U}_{[i]}^b)$ which shows that $a_i=n_iq_i$, for some $q_i\in \mathbf{N}$.  Set $B=\frac{c_aA^{q_i}\mathbf{U}_{[i]}^a}{U_i^{a_i}}.$   Notice that $\frac{B}{c_b\mathbf{U}_{[i]}^b},\frac{c_a\mathbf{U}_{[i]}^a}{B}\in R_{\mu}$ and we have

$$M=\iota(\frac{B}{c_b\mathbf{U}_{[i]}^b})\iota(\frac{c_a\mathbf{U}_{[i]}^a}{B})=
\iota(\frac{B}{c_b\mathbf{U}_{[i]}^b}){\iota(\frac{U_i^{n_i}}{A})}^{q_i}.$$
 By the induction hypothesis, the factor $\iota(\frac{B}{c_b\mathbf{U}_{[i]}^b})$ is algebraic over $\kappa_{\nu}$. Hence, we should only show that  $Z=\iota(\frac{U_i^{n_i}}{A})$ is algebraic over $\kappa_{\nu}$. Dividing both sides of  equation (\ref{precise ind key}) by $A^{p_i}$, we have (notice that by (\ref{index-power}): $m_i=n_ip_i$)
\begin{equation}\label{algebraic dependence}
Z^{p_i}+\iota(\frac{f_{i,n_i(p_i-1)}}{A})Z^{p_i-1}+\cdots+\iota(\frac{f_{i,n_i}}{A^{p_i-1}})Z+
\iota(\frac{f_{i,0}}{A^{p_i}})=\iota(\frac{U_{i+1}}{A^{p_i}})=0.
\end{equation}
Notice that, by the induction hypothesis, the coefficients of equation (\ref{algebraic dependence}) are algebraic over $\kappa_{\nu}$. Thus (\ref{algebraic dependence}) shows that $Z$ is algebraic over $\kappa_{\nu}$.

\item  If $L/K$ is of algebraic type and $\alpha=\omega$. By Theorem \ref{theorem:analytical irreducibility of valuation ring} the valuation ring $R_{\mu}$ is analytically irreducible, thus the valuation $\mu$ extends uniquely to a valuation (denoted again by $\mu$) to the $\frak{m}_{\mu}-$adic completion $\widehat{R_{\mu}}$. We construct a non-zero element $U_{\infty}\in\widehat{R_{\mu}}$ which is a  {\sl coefficient-wise  limit} for the sequence of the reduced key polynomials  $\{\overline{U}_i\}_{i\in\mathbf{N}}$ (Remark \ref{remark: af wei}.(ii)). There is some  $i_0\in\mathbf{N}$ such that for $i\geq i_0$ we have $m_i=1$ (Suppose this is not the case, so there are infinite number of $i$ such that $m_i>1$, but by  Definition \ref{weighted monomials}.(e) we have $\beta_{\omega}>(\prod_{i<\omega}m_i)\beta_1$. Thus, we have $\nu(U_{\omega})=\beta_{\omega}=\infty$ which is a contradiction).  For $i\geq i_0$, we set $\overline{U}_i=a_{i,N-1}x^{N-1}+\cdots+a_{i,1}x+a_{i,0}$, where $a_{i,t}\in K$  (Recall that $N$ is degree of the minimal polynomial of the element $x$ over $K$). For $i\geq i_0$, consider the equality $\overline{U}_{i+1}=\overline{U}_{i}+f_{i,0}$ (The reduced form of equation (\ref{ind key})). Let $f_{i,0}=\sum_{j}c_{i,j}\mathbf{\overline{U}}_{[i]}^{(j)}$ be the reduced $i-$adic expansion of $f_{i,0}$. As $m_{k}=1$, for $k\geq i_0$, the power of $\overline{U}_k$ is zero in any  adic monomial $\mathbf{\overline{U}}_{[i]}^{(j)}$ of $f_{i,0}$. So, for these  adic monomials  we have $\omega_{i}(\mathbf{\overline{U}}_{[i]}^{(j)})\leq\sum_{k<i_0}(n_k-1)\beta_k=\beta_{i_0}^{*}$. But, we have $\beta_i\to\infty(i\to\omega)$, so for any $j$, we have $\nu(c_{i,j})\to\infty(i\to \infty)$.  Now, for any $t\leq N-1$ we have $a_{i+1,t}-a_{i,t}\in \langle c_{i,j}\rangle_{j}$. This shows that for any $t\leq N-1$ the sequence $\{a_{i,t}\}_{i\in\mathbf{N}}$ is a Cauchy sequence for the $\nu-$adic topology in      $\widehat{R_{\nu}}$.   But, the ring $\widehat{R_{\nu}}$  is complete for the $\widehat{\frak{m}_{\nu}}-$adic topology, so by \cite{Te1}, Proposition 5.10, it is complete for the $\nu-$adic topology as well. Thus, we have $a_{\infty,t}:=\lim_{i\to\omega}a_{i,t}\in \widehat{R_{\nu}} \subset \widehat{R_{\mu}}$ is well-defined. In consequence, the element $U_{\infty}:=a_{\infty,N-1}x^{N-1}+\cdots+a_{\infty,1}x+a_{\infty,0}\in  \widehat{R_{\mu}}  $ is well-defined. Moreover, by Lemma \ref{lemma:invariance of minimal poly} the element $U_{\infty}$ is  non-zero, and thus $\mu(U_{\infty})$ is finite. By the construction of $U_{\infty}$, it is clear that for $i\geq i_0$ we have $\mathrm{in}_{\omega_i}(U_{\infty})=U_i$. So, we have $\omega_i(U_{\infty})=\beta_i$. This shows that $\mu(U_{\infty})\geq \lim_{i\to\omega}\beta_i=\infty$ which is a contradiction.
\end{itemize}
\end{proof}

\begin{remark}

The proof shows that in the case $L/K$ is of transcendental type we do not need the analytical irreducibility condition to meet the finiteness result.

\end{remark}

Here we give an example that shows that the analytical irreducibility is necessary in the case $L/K$ is of algebraic type\footnote{I am grateful to the referee for pointing out this example.}:
\begin{example}
Assume that $\mathrm{char}(k) \neq 2$. Let $K = k(y)$, $\nu$ the $y-$adic valuation. Let $L$ be the field of fractions of the integral
domain $k[x,y]/
(x^2-y^2-y^3)$. Then $\nu$ admits two extensions to $L$; their value groups can
both be identified with $\mathbf{Z}$, which we view as the value group of $\nu$. Let $\mu$ be the
extension characterized by the fact that $\mu(x + y) = 2$.  Let $\sqrt{1+y}=\sum_{i=0}^{\infty}b_iy^i$
 be
the Taylor expansion of  $\sqrt{1+y}$. Set $U_1=x$. Then the construction of the key polynomials gives an
infinite sequence $U_i = x+\sum_{j=1}^{i-1}b_{j-1}y^j$, for $i\geq 2$. One can show that $\mu(U_i)=\beta_i=i$, for $i\in\mathbf{N}$. This gives us an infinite sequence of key polynomials $\{U_i;\beta_i\}_{i\in \mathbf{N}}$ for the valuation $\mu$. There does not exist any finite subsequence of the key polynomials of this infinite sequence.

\end{example}

}

\section{Izumi's Theorem}\label{izumi theorem}
{For any two rank one valuations $\mu$ and $\mu'$ of a field $K$ with a common center in a subring $R$ of $K$, if there exists $c\in \mathbf{R}$ such that $\mu(y)\leq c\mu'(y)$, for any $y\in R$, then we write $\mu\leq c\mu'$. In such situation we define  $c_R(\mu,\mu')$ to be the minimum of such constants $c$; We call it the Izumi constant of the valuations $\mu$, $\mu'$. When the  ring $R$ is clear from the context we denote the Izumi constant by $c(\mu,\mu')$.

  Through this section $L$ is a field extension of a given field $K/k$, which is of the form $L=K(x)$, such that $L/K$ is either of transcendental type or algebraic type.
\begin{remark}\label{primitive}
The following are immediate from the definition of the Izumi constant.
\begin{itemize}
\item[(i)] If both   $\mu$ and $\mu'$ are centered  over $R$ with the ideal $\frak{p}$ as center  then $c_{R_{\frak{p}}}(\mu,\mu')$ exists provided that $c_R(\mu,\mu')$ exists. Moreover, we have $c_{R_{\frak{p}}}(\mu,\mu')=c_R(\mu,\mu')$.
\item[(ii)]
 For any three valuations $\omega,\omega',\omega''$ of a field $K$, such that all of them are centered over a ring $R\subset K$, if $c_R(\omega,\omega'')$ and $c_R(\omega'',\omega')$ exist then $c_R(\omega,\omega')$ exists and we have
 \begin{equation}\label{chain val}
 c_R(\omega,\omega')\leq c_R(\omega,\omega'')c_R(\omega'',\omega').
 \end{equation}
\end{itemize}
\end{remark}

 \begin{definition} Let $\nu$ be a valuation of $K$. For $\beta\in \mathbf{R}^{+}$, we define $\mathrm{ord}_{\nu,\beta}$ to be the Gaussian valuation extending $\nu$ to a valuation of $L$ with $\mathrm{ord}_{\nu,\beta}(x)=\beta$. In other words, for $f=\sum c_ix^i\in K[x]$ we have
$$\mathrm{ord}_{\nu,\beta}(f)=\mathrm{min}_{i}\{\nu(c_i)+i\beta\}.$$
In the case $L/K$ of algebraic type, this is a valuation of  $K(X)$ which we call it  a pseudo-valuation of $L$ (See Remark \ref{remark:af main}.(i)).
\end{definition}

\begin{theorem}\label{weighted izumi} Suppose $(L,\mu)$ is a valuation extending the divisorial valuation $(K,\nu)$. Assume that $\{U_i;\beta_i\}_{i\leq \alpha}$, $\alpha\in\mathbf{N}$, is a weighted basis of $K[x]$ with respect to $(K,\nu)$ such that,  with the notation of Definition \ref{weighted monomials}, we have $\mu=\omega_{\alpha}$. Moreover, assume that $\nu$ is centered over the local ring $R\subset K$. Then:
\begin{itemize}
\item[(i)]
We have $\omega_j(U_{i+1}^{\ell})= (\prod^{i}_{k=j}m_k).\ell\beta_j$, for  $\ell\in \mathbf{N}$, and  $j<i \leq \alpha$.

\item[(ii)]   For $j<i\in \mathbf{N}$, the Izumi constant $c_{R[x]}(\omega_{i+1},\omega_j)$ exists and we have $c_{R[x]}(\omega_{i+1},\omega_j)=\frac{\beta_{i+1}}{(\prod_{k=j}^{i}m_k).\beta_j}$.
\end{itemize}
\end{theorem}

\begin{proof}
For (i): We prove it for the case $\ell=1$. The general case is similar. Notice that  $U_i^{m_i}+U_i^{m_i-1}f_{i,m_i-1}+\cdots+U_if_{i,1}+f_{i,0}$ is the $i-$adic expansion of $U_{i+1}$. We have $\mathrm{deg}U_{i+1}=\mathrm{deg}U_{i}^{m_i}$ and $\mathrm{deg}f_{i,j}<\mathrm{deg}U_{i}$. On the other hand, in the algorithm of getting the $(i-1)-$adic expansion of $U_{i+1}$ from its $i-$adic expansion, the degree considerations shows that $U_{i-1}^{m_{i-1}m_i}$, which is generated in the first step of the algorithm, never cancels  in the process of the algorithm. In fact,  this is the unique monomial of degree equal to $\mathrm{deg}U_{i+1}$ in the $(i-1)-$adic expansion of $U_{i+1}$. Thus, the monomial $U_{i-1}^{m_{i-1}m_i}$ appears in the $(i-1)-$adic expansion of $U_{i+1}$; It has the least $\tilde{\omega}_{i-1}-$weight (because it appears starting from the first step of the algorithm). By induction, we reach to the following: The monomial $U_j^{m_j\cdots m_i}$ appears in the $j-$adic expansion of $U_{i+1}$; It has the least $\tilde{\omega}_j-$weight. Thus, we have proved $\omega_j(U_{i+1})=(\prod_{k=j}^{i}m_k)\beta_j$.

\par For (ii): First we prove the claim when $j=i$. Let us assume that $M=c_a\mathbf{U}_{[i+1]}^a\in K[U_1,\ldots,U_{i+1}]$ is a monomial of adic form. Then $\omega_{i+1}(M)=\nu(c_a)+\sum_{j=1}^{i+1}a_j\beta_j$. Suppose $M_1,\ldots,M_t$ is the sequence of $i-$expansions of $M$ generated in the algorithm of getting $i-$adic expansion of $M$ from its $i+1-$adic expansion. We have $\tilde{\omega}_i(M_1)=\nu(c_a)+\sum_{j=1}^{i}a_i\beta_i+a_{i+1}m_i\beta_i$. Set $\lambda=\nu(c_a)+\sum_{j=1}^{i}a_i\beta_i$. On the other hand $\omega_i(M)=\tilde{\omega}_i(M_t)\geq \tilde{\omega}_i(M_1)$. Thus, we have
$$\frac{\omega_{i+1}(M)}{\omega_i(M)}\leq \frac{\lambda+a_{i+1}\beta_{i+1}}{\lambda+m_ia_{i+1}\beta_i}\leq \frac{\beta_{i+1}}{m_i\beta_i}.$$
Now, assume that $f=\sum_j M_j$ is the $i+1-$adic expansion of an element $f\in R[x]$. Suppose that for $i_0$ we have $\omega_{i+1}(f)=\omega_{i+1}(M_{i_0})$.  Then  we have
$$\frac{\omega_{i+1}(f)}{\omega_{i}(f)}\leq \frac{\omega_{i+1}(M_{i_0})}{\omega_{i}(M_{i_0})}\leq \frac{\beta_{i+1}}{m_i\beta_i}.$$
This proves $c(\omega_{i+1},\omega_i)\leq \frac{\beta_{i+1}}{m_i\beta_i}$. On the other hand, by (i) we have $\frac{\omega_{i+1}(U_{i+1})}{\omega_i(U_{i+1})}=\frac{\beta_{i+1}}{m_i\beta_i}$ which shows that $c(\omega_{i+1},\omega_i)\geq \frac{\beta_{i+1}}{m_i\beta_i}$. Thus, we have $c(\omega_{i+1},\omega_i)= \frac{\beta_{i+1}}{m_i\beta_i}$.
\par For the general case, using  (\ref{chain val}) and the case $j=i$, we have $c(\omega_{i+1},\omega_j)\leq \prod_{k=j}^ic(\omega_{k+1},\omega_k)\leq \frac{\beta_{i+1}}{(\prod_{k=j}^{i}m_k).\beta_j}$.  But (i) shows that $c(\omega_{i+1},\omega_j)\geq \frac{\beta_{i+1}}{(\prod_{k=j}^{i}m_k).\beta_j}$. Hence, we have the equality.
\end{proof}

\begin{lemma}\label{prepare ord}
Suppose $\nu$ and $\nu'$ are valuations of a field $K$ and both are centered on the ring $R$ such that the Izumi constant $c_R(\nu,\nu')$ exists. Then the Izumi constant $c_{R[x]}(\mathrm{ord}_{\nu,\beta},\mathrm{ord}_{\nu',\beta'})$ exists and we have:
\begin{itemize}
 \item[(i)] We have $c_{R[x]}(\mathrm{ord}_{\nu,\beta},\mathrm{ord}_{\nu',\beta'})\leq \langle\frac{\beta}{\beta'}\rangle c_R(\nu,\nu')$, where
    \begin{equation}
    \langle\frac{\beta}{\beta'}\rangle=\left \{ \begin{array}{ll}
                                                 \frac{\beta}{\beta'}& \mathrm{when} \ \beta>\beta'\\
                                                 1& \mathrm{otherwise.}
                                                 \end{array}
                                                 \right.
    \end{equation}
\item[(ii)] Suppose $(K,\nu)\subset (L,\mu)$ then $c_{R[x]}(\mathrm{ord}_{\nu,\beta},\mu)\leq \langle\frac{\beta}{\mu(x)}\rangle$.

\end{itemize}
\end{lemma}

\begin{proof}
For (i): If $M=cx^{\ell}\in R[x]$ is a monomial then one can easily check that $\mathrm{ord}_{\nu,\beta}(M)\leq \langle\frac{\beta}{\beta'}\rangle c(\nu,\nu')\mathrm{ord}_{\nu',\beta'}(M)$. Now, assume that $f=\sum_i M_i\in R[x]$. Suppose that $\mathrm{ord}_{\nu,\beta}(f)=\mathrm{ord}_{\nu,\beta}(M_0)$ and $\mathrm{ord}_{\nu',\beta'}(f)=\mathrm{ord}_{\nu',\beta'}(M_1)$. Then we have $$\mathrm{ord}_{\nu,\beta}(f)=\mathrm{ord}_{\nu,\beta}(M_0)\leq \mathrm{ord}_{\nu,\beta}(M_1)\leq \langle\frac{\beta}{\beta'}\rangle c(\nu,\nu')\mathrm{ord}_{\nu',\beta'}(M_1),$$ which shows that $\mathrm{ord}_{\nu,\beta}(f)\leq \langle\frac{\beta}{\beta'}\rangle c(\nu,\nu')\mathrm{ord}_{\nu',\beta'}(f)$.

 \par For (ii): Set $\mu(x)=\beta'$.  It is clear that $c(\mathrm{ord}_{\nu,\beta'},\mu)= 1$. By (i) and Remark \ref{primitive}.(ii) we have $$c(\mathrm{ord}_{\nu,\beta},\mu)\leq c(\mathrm{ord}_{\nu,\beta},\mathrm{ord}_{\nu,\beta'}) c(\mathrm{ord}_{\nu,\beta'},\mu)\leq \langle\frac{\beta}{\beta'}\rangle.$$
\end{proof}

\begin{theorem} \label{inductive izumi}Suppose that $\mu$ and $\mu'$ are two divisorial valuations of the field  $L=K(x)$. Suppose $\{U_i;\beta_i\}_{i=1}^n$, where $n\in \mathbf{N}$ is a weighted basis of $K[x]$ such that $\omega_n=\mu$. Assume that $\mu_{\mid_{K}}=\nu$, $\mu'_{\mid_{K}}=\nu'$, and $\nu$ and $\nu'$ are both centered over a ring $R$, and  both $\mu$ and $\mu'$ are centered over the ring $R[x]$. Moreover, suppose $c_R(\nu,\nu')\in \mathbf{R}^{+}$ exists. Then  $c_{R[x]}(\mu,\mu')$ exists and we have
$$c_{R[x]}(\mu,\mu')\leq \mathrm{max}\{\frac{1}{\mu(x)},\frac{1}{\mu'(x)}\} c_R(\nu,\nu')\frac{\mu(U_n)}{\mathrm{deg}U_n}.$$

\end{theorem}
\begin{proof}
Set $\beta=\mu(x)$ and $\beta'=\mu'(x)$. By (iii) of Theorem \ref{weighted izumi} and Lemma \ref{prepare ord}, the three  Izumi constants $c_{R[x]}(\mu,\mathrm{ord}_{\nu,\beta})$, $c_{R[x]}(\mathrm{ord}_{\nu,\beta},\mathrm{ord}_{\nu',\beta'})$, and $c_{R[x]}(\mathrm{ord}_{\nu',\beta'},\mu')$ exist. So, by Remark \ref{primitive}.(ii)  the Izumi constant  $c_{R[x]}(\mu,\mu')$ exists and we have
\begin{center}
\begin{tabular}{ll}

$c(\mu,\mu')$&$\leq c(\mu,\mathrm{ord}_{\nu,\beta})c(\mathrm{ord}_{\nu,\beta},\mathrm{ord}_{\nu',\beta'})c(\mathrm{ord}_{\nu',\beta'},\mu')$\\ \\
&$\leq \frac{\mu(U_n)}{\mathrm{deg}U_n.\beta}\langle\frac{\beta}{\beta'}\rangle c(\nu,\nu').$\\  \\
\end{tabular}
\end{center}
Now it is sufficient to note that $\langle\frac{\beta}{\beta'}\rangle\frac{1}{\beta}=\mathrm{max}\{\frac{1}{\mu(x)},\frac{1}{\mu'(x)}\}$.
\end{proof}

\begin{proof}
 [Proof of Theorem \ref{final izumi}]  By Corollary \ref{corollay:uniqueness of extension to completion}  we can assume that $R$ is complete. By Cohen's structure theorem  we have $R\cong k[[X_1,\ldots,X_n]]/I$. We prove the result by induction on $n$. Set $S=k[[X_1,\ldots,X_{n-1}]]/(I\cap k[[X_1,\ldots,X_{n-1}]])$. First, notice that if $I=\langle f_i\rangle_{i=0,\ldots,d}$, after a polynomial change of coordinates (if necessary), we can assume that $f_i(0,\ldots,0,X_n)\neq 0$ for any $i\leq d$. Now, by Weirestrass' preparation, we can assume that $f_i\in k[[X_1,\ldots,X_{n-1}]][X_n]$. This shows that $R \cong \widehat{(S[X_n]/(I\cap S[X_n]))}$, where the completion is taken with respect to the maximal ideal of the origin. Set $R_1=S[X_n]/(I\cap S[X_n])$. By Corollary \ref{corollay:uniqueness of extension to completion}, to show that $c_R(\mu,\mu')$ exists, it is sufficient to show that  $c_{R_1}({\mu\mid}_{R_1}, {\mu'\mid}_{R_1})$ exists (In fact, by the same corollary we have  $c_R(\mu,\mu')=c_{R_1}(\mu\mid_{R_1}, \mu'\mid_{R_1})$). But, by the induction hypothesis the Izumi constant $c_S({\mu\mid}_{S}, {\mu'\mid}_{S})$ exists. So, by Theorem \ref{inductive izumi}  the Izumi constant $c_{R_1}({\mu\mid}_{R_1}, {\mu'\mid}_{R_1})$ exists.
 \end{proof}
\begin{remark}
  Instead of considering $\mathbf{R}$ as the  totally  ordered group that contains all the value groups of divisorial valuations, one can fix a copy of $\mathbf{Z}$ as the value group of all valuations (this is the assumption of \cite{Swan}). In this situation, as $\mu(x)\geq 1$ for any $(L,\mu)$,  we can make the bound of $c_{R[x]}(\mu,\mu')$ sharper;  This bound does not depend on $\mu'$. In this case $c_{R[x]}(\mu,\mu')\leq c_R(\nu,\nu')\frac{\mu(U_n)}{\mathrm{deg}U_n}.$
\end{remark}

}

\bibliography{valbib}
\bibliographystyle{plain}
\end{document}